\documentclass[12pt]{amsart}

\usepackage{times, amsmath, amsfonts, amssymb, amsthm}
\usepackage{times}
\usepackage{srcltx}
\usepackage{setspace}
\usepackage{color}

\usepackage[utf8]{inputenc}

\usepackage[left=3.5cm, right=3.5cm]{geometry}

\newcommand{\dd}{\,\text{\rm d}}             % a straight d for differentials
\newcommand{\ee}{\text{\rm e}}

\newcommand{\abs}[1]{\left| #1 \right|}

\theoremstyle{plain}
\newtheorem{theorem}{Theorem}[section]
\newtheorem{lemma}[theorem]{Lemma}
\newtheorem{proposition}[theorem]{Proposition}
\newtheorem{corollary}[theorem]{Corollary}

\theoremstyle{definition}
\newtheorem{remark}[theorem]{Remark}

\newtheorem{assu}{Assumption}

\newcounter{mycount}

\newenvironment{romlist}{\begin{list}{\it \roman{mycount})}%
   {\usecounter{mycount}\labelwidth=1cm\itemsep 0pt}}{\end{list}}

\renewenvironment{proof}[1][] {\noindent {\bf Proof#1.} }{\hspace*{\fill}$\square$\medskip\par}

\newcommand{\N}{{\mathbf N}}
\newcommand{\Z}{{\mathbf Z}}
\newcommand{\R}{{\mathbf R}}

\newcommand{\E}{{\mathbf E}}
\newcommand{\F}{{\mathcal F}}

\renewcommand{\P}{{\mathbf P}}

\newcommand{\supp}{\operatorname{supp}}

 \title[Positive speed in QEW]{Positive speed of propagation in a semilinear parabolic interface model with unbounded random coefficients}

\date{\today}

\author[Dondl]{P.\ W.\ Dondl}
\address[P.\ W.\ Dondl]{Institute for Applied Mathematics, Universit\"at Heidelberg, Im Neuenheimer Feld 294, D-69120 Heidelberg, Germany}
\email{pwd@hcm.uni-bonn.de}
\urladdr{http://www.dondl.net/}

\author[Scheutzow]{M.\ Scheutzow}
\address[M.\ Scheutzow]{Fakult\"at II, Institut f\"ur Mathematik, Sekr.\ MA 7--5,
Technische Universit\"at Berlin, Strasse des 17.\ Juni 136,
D-10623 Berlin, Germany} \email{ms@math.tu-berlin.de}
\urladdr{http://www.math.tu-berlin.de/$\sim$scheutzow/}

\begin{document}  
\begin{abstract}
We consider a model for the propagation of a driven interface through a random field of obstacles. The evolution
equation, commonly referred to as the Quenched Edwards-Wilkinson model, is a semilinear parabolic equation
with a constant driving term and random nonlinearity  to model the influence of the obstacle field. For the case
of isolated obstacles centered on lattice points and admitting a random strength with exponential tails, we show
that the interface propagates with a finite velocity for sufficiently large driving force.  The proof consists
of a discretization of the evolution equation and a supermartingale estimate akin to the study of branching random
walks.
\end{abstract}
\maketitle

\section{Introduction, model, and the main result}

In this article, we consider a parabolic model for the evolution of an interface in a random medium. The interface at
time $t$ is assumed to be the graph of a function. The local
velocity of the interface is governed by line tension and a competition between a constant external driving force $F>0$
and a heterogeneous random field $f \colon \R\times \R \times \Omega \to \R$. This field describes the environment of the interface. More
precisely, let $(\Omega,\F,\P)$ be a complete probability space. 
%$\omega \in \Omega$. 
We consider the evolution equation
\begin{align}
\label{eq:governing}
u_t(x,t,\omega) &=  u_{xx}(x,t,\omega) - f(x,u(x,t,\omega),\omega) + F \\
u(x,0,\omega) &= 0 \nonumber
\end{align}
for $t\ge0$, $x\in\R$, $\omega \in \Omega$. The heterogeneous field $f\ge0$ thus plays the role of
obstacles that impede the free propagation of the interface.

We are particularly interested in the macroscopic behavior of solutions to~\eqref{eq:governing} and their dependence on the
parameter $F$. Specifically, assume that the random field $f$ is not uniformly bounded from above, i.e., there exist obstacles
of arbitrarily large strength. Can one now find a deterministic constant $F^*$, such that the interface will propagate to 
infinity for $F\ge F^*$? The question of non-existence of a stationary solution in such a model -- for obstacles with exponential 
tails in the distribution of their strength -- was answered by Coville, Dirr, and Luckhaus in~\cite{Coville:2010p1074}. 
The present article extends this result and proves finite speed of propagation for large enough driving force in the same setting.

In \cite{Dirr:2009p1247}, the opposite question
was answered, namely whether interfaces become stuck (i.e., whether non-negative 
stationary solutions to the evolution equation~\eqref{eq:governing}
exist) for small but positive external load $F$. The work there is based on a percolation result~\cite{Dirr:2010p645}. Together, the results show a transition from a viscous kinetic relation at the
microscopic level to a stick-slip behavior, leading to a rate independent hysteresis at the macro-level. Problems of the form
of~\eqref{eq:governing} find substantial interest in the physics community, see for example~\cite{Kardar:1998p151,Brazovskii:2004p1248}.
Further connections to physics and to homogenization problems in degenerate elliptic equations can be found in~\cite{Coville:2010p1074}.

In order to introduce the main result of this article, we first fix the nature of the random field $f$.
\begin{assu}
\label{ass:f}
Let $f_{ij}(\omega)$, $i,j \in \Z$, be iid (independent and identically distributed) non-negative random
variables with a finite exponential moment, i.e., $\E \exp\{\lambda f_{00}(\omega)\} = \beta <\infty$ for some $\lambda >0$. Furthermore,
set $\phi \colon \R^2 \to \R$, $\phi \in C^1(\R^2)$ such that $0\le \phi \le 1$ and 
\begin{equation} \label{eq:def_delta}
\supp \phi \subset [-\delta, \delta]^2 \quad \textrm{with $\delta < 1/2$}.
\end{equation}
The random field $f$ shall then be given as
\begin{equation}
\label{eq:random_field}
f(x,y,\omega) := \sum_{i,j\in \Z} f_{ij}(\omega) \phi(x-i, y-j-1/2).
\end{equation}
\end{assu}
This means that the random field consists of obstacles centered at points of the square lattice (shifted by 1/2 in the $y$-(propagation)
direction) with random strength with exponential tail of the distribution, but uniform shape\footnote{As one can easily see from
the proofs, the assumption of uniformity of $\phi$ can be relaxed, as long as certain obvious uniform bounds are still adhered to.}. 
We show the following result about the propagation velocity of the interface.
\begin{theorem}
\label{thm:cont}
Let $u \colon \R\times [0, \infty)\times \Omega \to \R$ be a solution to~\eqref{eq:governing}  and $f$ as 
in Assumption~\ref{ass:f}. Let
\begin{align*}
U \colon [0,\infty) \times \Omega &\to \R \\
U \colon (t,\omega) &\mapsto \int_0^1 u(\xi, t, \omega) \dd \xi.
\end{align*}
Then there exists a non-decreasing function $V \colon [0,\infty) \to [0,\infty)$, that is not identically zero and
only depends on the parameters $\lambda$, $\beta$ and $\delta$, such that
\begin{equation}
\E \frac{U(t)}{t} \ge V(F) \quad \textrm{for all $t>0$}.
\end{equation}
\end{theorem}
\begin{remark} \label{rem:V_choice}
It will be evident from the proof that a viable choice for $V$ is
$$
V(F)= \frac{1}{4(1+\delta)}\sup_{\mu>\tilde{\lambda}} \frac 1 {\mu} \Big(\tilde{\lambda} ((1-2\delta)F - 2) -\log p(\tilde{\lambda},\mu)-\log \tilde{\beta} \Big),
$$
where
$$
p(l,m):= \frac 1{1-\ee^{-l}} + \frac 1{1-\ee^{l-m}}.
$$
The parameter $\tilde{\lambda}$ can be chosen arbitrarily with $\tilde{\lambda}\in (0,\lambda)$ and the parameter $\tilde{\beta}$ is given 
as
$$
\tilde{\beta} := \ee^{\tilde{\lambda}} \inf_{c>\exp\{(\log \beta)\frac{180\tilde{\lambda}}{\lambda}\}}\left(c+\int_c^\infty  
\frac{ \beta \ee^{-\frac{\lambda}{180\tilde{\lambda}}\log x }}{1-\beta\ee^{- \frac{\lambda}{180\tilde{\lambda}} \log x }} \dd x\right).
$$
\end{remark}
\begin{remark}
Theorem~\ref{thm:cont} includes the result of non-existence of non-negative stationary solutions by Coville, Dirr and Luckhaus~\cite{Coville:2010p1074}, since existence of a non-negative stationary
solution to~\eqref{eq:governing} would violate the fact that $\limsup_{t \to \infty}u(t,x)/t \ge V(F)$ for all $x\in\R$ almost
surely. See Subsection~\ref{subsec:almost_cont} for a proof of this corollary of Theorem~\ref{thm:cont}.
\end{remark}

Due to our stationarity and independence assumption on $f$, it is clear that the processes
$U_i(t,\omega) := \int_i^{i+1} u(\xi, t, \omega) \dd \xi$, $i\in\Z$ are stationary and ergodic for each $t \ge 0$.
Therefore, this theorem shows that the average area covered by the interface per
unit time and per unit length is bounded from below by a positive constant deterministic 
velocity if the driving force is sufficiently large.

We will prove the Theorem~\ref{thm:cont} by first addressing a discretized version of the evolution equation in
Section~\ref{sec:discrete} and then showing the reduction of the continuum problem to the discrete problem
in Section~\ref{sec:cont}. In the final Section~\ref{sec:conclusion} we conclude by presenting a number of open questions.

\section{The discretized front propagation model} \label{sec:discrete}
We consider the following discrete model, which arises as a lower bound for the velocity in the continuum
problem~\eqref{eq:governing} -- as shown in Section~\ref{sec:cont} -- but also is interesting to study in its own right.
Let $u_i \colon [0,\infty) \to \R$ solve the equation
\begin{align}
\frac {\dd}{\dd t} u_i(t)&=\Big( u_{i-1}(t) +  u_{i+1}(t) -2  u_{i}(t)-\tilde{f}_i(u_i(t),\omega)+F\Big)^+ \label{eq:disc_evol} \\
u_i(0)&=0, \qquad i \in \Z, \nonumber
\end{align}
where $F \ge 0$ and $\tilde{f}_i:\R \times \Omega \to [0,\infty)$, $i \in \Z$ are independent and identically distributed functions 
such that the map $(y,\omega) \mapsto \tilde{f}_0(y,\omega)$ is measurable with respect to the product of the Borel-$\sigma$ 
algebra on $\R$ and $\F$ and the map $y \mapsto \tilde{f}_0(y,\omega)$ is locally bounded for almost all $\omega \in \Omega$.
These assumptions guarantee that the equation above has a unique solution which depends measurably on $\omega$ 
for each $t \ge 0$. Note that we do not assume that the map $y \mapsto \tilde f_0(y)$ is stationary.
\begin{remark}
The symbol $(\cdot)^+$ denotes taking the non-negative part of the term inside the parenthesis. One can easily see from
the comparison principle for (discrete) elliptic equations that taking the non-negative part is only necessary if one can not ensure that
the initial velocity is non-negative. For the continuous equation, non-negativity is shown in Proposition~\ref{prop:pos_dudt}.
\end{remark}

\subsection{Lower bound on the averaged velocity} \label{subsec:averaged}

For this discretized version of the main evolution problem~\eqref{eq:governing} we can show the analog to Theorem~\ref{thm:cont}.
\begin{theorem}\label{main}
Assume -- in addition -- that there exists $\tilde{\lambda}>0$ such that 
$$
\tilde{\beta}:=\sup_{n \in \Z} \E \sup_{n-.5 < y \le n+.5} \exp\{\tilde{\lambda}  f_0(y,\omega)\} < \infty.
$$
Then there exists a non-decreasing function $W:[0,\infty) \to [0,\infty)$ which is not identically zero and 
which depends on $\tilde{\lambda}$ and $\tilde{\beta}$ only, such that 
$$
\E \dot u_0(t) \ge W(F)\; \mbox{ for all } t\ge 0 \mbox{ and hence } \E \frac{u_0(t)} t \ge W(F)\; \mbox{ for all } t> 0. 
$$ 
Specifically, we can choose 
$$
W(F)=\sup_{\mu>\tilde{\lambda}} \frac {1}{\mu} \Big(\tilde{\lambda} (F - 2) -\log p(\tilde{\lambda},\mu)-\log \tilde{\beta} \Big),
$$
where $p$ is as in Remark~\ref{rem:V_choice}. In fact, the function $V$ there is just a rescaled version of $W$.
%Under the following more restrictive conditions on $f_0$ ({\bf noch zu tun, oder gilt das immer?}) we 
%also have 
%$$
%\liminf_{t \to \infty} \frac{u_0(t)} t \ge V(F) \mbox{ a.s.}
%$$
\end{theorem}

\begin{remark}
The supremum inside the expectation in the definition of $\tilde{\beta}$ is not necessarily measurable. Strictly speaking, one should replace the 
expectation by the infimum of the expectation of all random variables dominating the supremum.
\end{remark}

%\comment{
%\begin{remark}
%The final statement in Theorem \ref{main} typically does not hold for example if the $f_i$ do not depend on $y$.
%(Noch zu tun, stimmt das?). 
%
%Ich denke, dass auch in dem Fall die Aussage korrekt ist (pwd)
%\end{remark}}

We will prove Theorem \ref{main} using the following result:

\begin{lemma}\label{discrete}
Let $\bar f_{ij}:\Omega \to [0,\infty)$, $i,j \in \Z$ be random variables such that the functions 
$\bar f_i:\Omega \times \Z \to [0,\infty)$ defined as $\bar f_i(\omega,j):=\bar f_{ij}(\omega)$ are independent. 
Assume that there exists some $\bar{\lambda}>0$ such that $\bar \beta:=\sup_{m,n \in \Z} \E \exp\{\bar{\lambda} \bar f_{mn}\}< \infty$. 
Then, for each $F>0$, there exists a set $\Omega_0$ of full measure such that for any function $w: \Omega \times \Z \to \N_0$ and any 
$\omega \in \Omega_0$, we have 
\begin{equation}\label{ineq}
\limsup_{n \to \infty} \frac 1n \sum_{i=1}^n \Big(w_{i-1}+w_{i+1}-2 w_i-\bar f_i(\omega,w_i)+F\Big)^+ \ge \bar{W}(F),
\end{equation}
where 
$$
\bar{W}(F):=\sup_{\mu>\bar{\lambda}} \frac 1 {\mu} \Big(\bar{\lambda} F -\log p(\bar{\lambda},\mu)-\log \bar{\beta} \Big),
$$
and $p$ is defined as in Remark~\ref{rem:V_choice}.
\end{lemma}

\begin{proof}
Since the assertion holds trivially in case $\bar{W}(F)\le 0$, we can and will assume that $\bar{W}(F)>0$. In this case, the supremum 
in the definition 
of $\bar{W}(F)$ is actually a maximum (observing that the function of $\mu$ converges to 0 as $\mu \to \infty$) and we choose a maximizer 
which we denote again by $\mu$.   Further, 
it suffices to show that for each fixed $w_{-1},w_0 \in \N_0$ there exists a set $\Omega_0$ of full measure such that for any function 
$w:\Omega \times \Z \to \Z$ attaining the prescribed values at $-1$ and $0$, and all $\omega \in \Omega_0$ satisfying 
$$
\limsup_{n \to \infty} \frac 1n \sum_{i=0}^{n-1} \Big(w_{i-1}+w_{i+1}-2 w_i-\bar f_i(\omega,w_i)+F\Big)^+ < \bar{W}(F) \mbox{ on } \Omega_0,
$$
we have $\liminf_{n \to \infty} w_n(\omega)<0$.

To see that this is true, we proceed as follows. For $n \in \N_0$ we call each function $w:\{-1,0,...,n\} \to \Z$ resp. 
$w:\{-1,0,...\} \to \Z$ starting with the 
prescribed values $w_{-1},w_0$ a {\em path} of length $n$ resp. a {\em path}. To each path (of length $n$), we 
associate $v_n:=w_n-w_{n-1}$ and 
$$
s_n:=\sum_{i=0}^{n-1} \Big(w_{i-1}+w_{i+1}-2 w_i-\bar f_i(\omega,w_i)+F\Big)^+,\qquad n \in \N_0. 
$$
Define
$$
Y_n:=\sum \exp\{\bar{\lambda} v_n-\mu s_n\},
$$
where the sum is taken over all paths of length $n$. Let $\F_k$ denote the $\sigma$-algebra generated by $\bar f_0,...,\bar f_{k-1}$. 
We now claim that
$$
\E\big(Y_{n+1}|\F_n\big) \le \gamma Y_n \mbox{ a.s.},
$$
where  
$$
\gamma=\bar \beta \exp\{-\bar{\lambda} F\} \Big( \frac 1{1-\ee^{-\bar{\lambda}}} + \frac 1{1-\ee^{\bar{\lambda}-\mu}}\Big)
$$ 
showing that $Z_n:=Y_n/\gamma^n$ is a (nonnegative) supermartingale. The proof of the claim is a straightforward computation:
\begin{align*}
\E\big(&Y_{n+1}|\F_n\big)\\
&=\sum  \exp\{-\mu s_n\}\E\Big( \sum_{j \in \Z}\exp\{\bar{\lambda} j-\mu(j-v_n-\bar  f_n(\omega,w_n)+F)^+  \}|\F_n \Big)\\
&=\sum  \exp\{-\mu s_n\}\E\Big( \sum_{j \ge \lceil v_n+\bar  f_n(\omega,w_n)-F\rceil}
\exp\{(\bar{\lambda}-\mu) j + \mu(v_n+\bar  f_n(\omega,w_n)-F)  \}\\&\hspace{.5cm}+ \sum_{j < \lceil v_n+\bar  f_n(\omega,w_n)-F\rceil} \exp\{\bar{\lambda} j\}   |\F_n \Big)\\
&=\sum  \ee^{-\mu s_n} \E\Big( \frac{\exp\{(\bar{\lambda}-\mu)\lceil v_n+\bar  f_n(\omega,w_n)-F\rceil\}}{1-\exp\{\bar{\lambda} - \mu\}} 
  \exp\{\mu (v_n+\bar  f_n(\omega,w_n)-F)\} \\
&\hspace{.5cm}+ \exp\{\bar{\lambda} \big(\lceil v_n+\bar  f_n(\omega,w_n)-F\rceil-1\big)\}\frac 1{1-\exp\{-\bar{\lambda}\}}|\F_n \Big)\\
&\le \sum  \ee^{-\mu s_n} \E\Big(  \exp\{\bar{\lambda} \big( v_n+\bar  f_n(\omega,w_n)-F)\} \Big(\frac 1{1-\ee^{\bar{\lambda}-\mu}} 
+ \frac 1{1-\ee^{-\bar{\lambda}}}        \Big)      |\F_n \Big)\\     
&\le \gamma Y_n,
\end{align*}
where the first sum is extended over all paths of length $n$. The operator $\lceil a \rceil = \inf \{z\in\Z : z\ge a\}$ denotes taking
the integer ceiling of $a\in\R$.

By the supermartingale convergence theorem, 
there exists a set $\Omega_0$ of full measure such that $\sup_{n \in \N_0} Y_n/\gamma^n$ is finite for all $\omega \in \Omega_0$. 
On $\Omega_0$, we therefore have
$$
\limsup_{n \to \infty} \frac 1n \sup \{\bar{\lambda} v_n - \mu s_n\}\le \limsup_{n \to \infty} \frac 1n \log Y_n \le \log \gamma, 
$$
where the sup is extended over all paths of length $n$. Therefore, for each path
$$
\bar{\lambda} \limsup_{n \to \infty} \frac {v_n}n < \log \gamma + \mu \bar{W}(F) = 0  \mbox{ on the set } 
\Big\{\limsup_{n \to \infty} \frac{s_n}n < \bar{W}(F)\Big\} \cap \Omega_0.
$$
%If $w$ is as in the statement of the theorem, i.e.~it takes only nonnegative values, then left hand side of the inequality is 
%nonnegative which implies that the event on the right hand side has measure zero, so the statement of the theorem follows.
In particular, we have $\liminf_{n \to \infty} w_n(\omega)<0$ on that set (in fact even for $\limsup$ instead of $\liminf$), 
so the statement of the lemma follows.
\end{proof}

\begin{remark}
We note the following properties of the functions $W$ and $\bar{W}$.
\begin{romlist} 
\item The functions $W$ and $\bar{W}$ defined in Theorem~\ref{main} and Lemma~\ref{discrete} are nonnegative (let $\mu \to \infty$). 
Further, $\bar{W}(F) >0$ whenever  $F> \frac 1{\bar{\lambda}} \big( \log \bar \beta + \log \big( 1+(1-\exp\{-\bar{\lambda}\})^{-1}\big)\big)$, 
$W(F)>0$ whenever $F-2$ satisfies the same property.

\item If we choose $\mu=\bar{\lambda} + 1/F$, then we see that $F-\bar W (F) \lesssim \frac{1}{\bar{\lambda}} \log F$ as $F\to \infty$.
\end{romlist}
\end{remark}

\begin{remark}
The set $\Omega_0$ in the previous lemma can actually be chosen independently of $F$ since both sides of \eqref{ineq} depend 
continuously on $F$.
\end{remark}

Now we turn to Theorem \ref{main}. The proof will follow immediately from the following statement.
\begin{lemma} \label{lem:vel_est}
Let $\tilde{f}_i$ and $W$ as in the statement of Theorem~\ref{main}. Then almost surely we have for all
non-negative sequences $u_i \in \R$, $u_i \ge 0$ for all $i \in \Z$ that
$$
 \limsup_{n \to \infty}\frac 1n \sum_{i=0}^{n-1} \Big(u_{i-1}+u_{i+1}-2u_i - \tilde{f}_i(u_i,\omega) +F \Big)^+ \ge W(F).
$$
\end{lemma}
\begin{proof}
We define the associated discrete path $w_i$, $i \in \{-1,0,...\}$ taking values in $\N_0$ by rounding $u_i$ to the closest integer 
(rounding up in case of ties). 
To apply Lemma \ref{discrete}, we define $\bar f_{ij}:=\sup_{y \in (j-.5,j+.5]}  \tilde{f}_i(y,\omega)$.
%(we add 2 in order to compensate the deviations between the $u_i$ and the $w_i$). 
Then %-- by Theorem \ref{discrete} -- 
we have
$$
\Big( u_{i-1} +  u_{i+1}-2  u_{i}-\tilde{f}_i(u_i,\omega)+F\Big)^+  \ge \Big( w_{i-1} +  w_{i+1} -2  w_{i}-\bar f_i(w_i,\omega)+F-2\Big)^+  
$$ 
(we subtract 2 in order to compensate the deviations between the $u_i$ and the $w_i$) and therefore -- by Lemma \ref{discrete} -- we obtain
$$
\limsup_{n \to \infty} \frac 1n \sum_{i=0}^{n-1} \Big(u_{i-1}+u_{i+1}-2u_i - \tilde{f}_i(u_i,\omega) +F \Big)^+ \ge \bar W(F-2)=W(F)
$$
for $\omega \in \Omega_0$ and $\Omega_0$ of full measure and independent of the choice of the $u_i$.
\end{proof}
\noindent{\bf Proof of Theorem \ref{main}.}
Assume that the (first) statement in the theorem is untrue. Then there exist $F \ge 0$ and some $t_0$  
such that $\E\dot u_0(t_0) < W(F)$. By our stationarity and independence assumptions on the field $f$, the processes 
$u_i(t_0)$, $\dot u_i(t_0)$, $i \in \Z$ are stationary and ergodic and take values in $[0,\infty)$. 
We write $u_i$ instead of $u_i(t_0)$. By Birkhoff's ergodic theorem, we have 
$\E\dot u_0= \lim_{n \to \infty}\frac 1n \sum_{i=0}^{n-1} \Big(u_{i-1}+u_{i+1}-2u_i - \tilde{f}_i(u_i,\omega) +F \Big)^+ < W(F)$ almost surely.  
This is a contradiction to Lemma~\ref{lem:vel_est}.
\hfill $\Box$

\subsection{Almost sure statements about the propagation velocity} \label{subsec:almost}
In the discrete case, it is also possible to show some statements about the velocity of the interface that hold almost surely. 

\begin{proposition} \label{prop:diff_est} Consider our standard discrete set-up from Theorem~\ref{main} with the following relaxed assumptions on the 
$\tilde{f}_i$: the $\tilde{f}_i$ are nonnegative random functions 
%such that $0 \le f_i(y,\omega) \le C_i(\omega)(\log^+(y)+1)$, where the $C_i(\omega)$ are arbitrary $[0,\infty)-$valued random variables 
(no independence or stationarity assumptions). Then we have $\lim_{t \to \infty}\frac{u_i(t)-u_j(t)}{t}=0$ for all $i,j \in \Z$ and all 
$\omega \in \Omega$. 
%(or almost all $\omega$ if we allow the $C_i(\omega)$ to take the value $\infty$ on a set of meausure 0).
\end{proposition}

%\begin{proof}
%Let $\alpha >0$ and assume that we have for some $\omega$, $t_0>0$ and $h>0$ we have 
%$$
%\frac{u_{-1}(t)+u_1(t)-2u_0(t)}{t} \ge \alpha
%$$
%for all $t \in [t_0,t_0+h]$. Then, for  $t \in [t_0,t_0+h]$,
%\begin{align*}
%\dot u_0(t)&=\big(u_{-1}(t)+u_1(t)-2u_0(t)-f_i(u_0(t),\omega)+F \big)^+\\
%&\ge \alpha t - C_0(\omega)(\log^+(F\,t)+1)+F
%\end{align*}
%and therefore
%$$
%u_0(t+h) \ge u_0(t_0) + \alpha \Big( ht_0+\frac 12 h^2\Big) - h\,  C_0(\omega)(\log^+(F\,(t_0+h)+1)+hF.
%$$
%Dividing by $t_0+h$ and observing that $F \ge u_0(t)/t \ge 0$ for all $t>0$, we see that for each choice of $alpha,h>0$ which satisfies 
%$\alpha h >F$ there exists some $t_1=t_1(\omega)$, such that $\frac{u_{-1}(t)+u_1(t)-2u_0(t)}{t}$ cannot 
%stay above the level $\alpha$ in any subinterval of $[t_1,\infty)$ of length $h$. 
\begin{proof}
Let 
$$
H(t):=u_{i-1}(t) + u_{i+1}(t) - 2u_i(t) + F.
$$
We then have
$$
H(t) \ge h(t):= u_{i-1}(t) + u_{i+1}(t) - 2u_i(t) + F -\tilde{f}_i(u_i(t),\omega).
$$
If $h(t) < 0$, we have $\frac{\dd}{\dd t} u_i(t) =0$ and thus $\frac{\dd}{\dd t} H(t) = \frac{\dd}{\dd t} h(t) \ge 0$.
Since $H(t)$ is a continuous function of time, we have $H(t) \ge 0$ for all $t$.

Now let $\gamma>0$. The non-negativity of $H(t)$ implies that
\begin{equation}\label{gamma}
u_{i-1}(t)+u_{i+1}(t)-2u_i(t)>-\gamma t \qquad \mbox{ whenever } t > \frac{F}{\gamma},\;i \in \Z. 
\end{equation}
Fix $\Gamma>0$ and $n \in \N$ and define $\gamma:=\Gamma/n$. Let $t>F/\gamma$ and assume that there exists some $i \in \Z$ 
such that $\frac{u_{i+1}(t)-u_i(t)}{t}\ge \Gamma$. Since $F \ge u_j(t)/t \ge 0$ for all $j$, \eqref{gamma} implies that
$$
F \ge \frac{u_{i+n}(t)-u_i(t)}{t} = \sum_{k=0}^{n-1} \frac{u_{i+k+1}(t)-u_{i+k}(t)}{t}\ge \sum_{k=0}^n (\Gamma-k\gamma)=\Gamma \frac{n+1}{2}.
$$
This inequality can only hold in case $n \le \frac{2F}{\Gamma}-1$. The same is true in case  $\frac{u_{i+1}(t)-u_i(t)}{t}\le -\Gamma$. 
Therefore, $\frac{|u_{i+1}(t)-u_i(t)|}{t}\le \Gamma$ for all $i \in \Z$ whenever $t> 2F^2/\Gamma^2$. Since $\Gamma>0$ was arbitrary, 
the assertion follows.
\end{proof}

\begin{remark} \label{rem:limsup}
Under the assumptions of Theorem  \ref{main}, the processes 
\begin{align*}
i &\mapsto \limsup_{t \to \infty}u_i(t)/t \quad \textrm{and} \\
i &\mapsto \liminf_{t \to \infty}u_i(t)/t
\end{align*}
 are ergodic. By the previous proposition, these processes do not depend on $i$. These two facts 
together imply that there exist deterministic numbers $0 \le c_1 \le c_2 \le F$ such that $\liminf_{t \to \infty}u_i(t)/t=c_1$ and 
$\limsup_{t \to \infty}u_i(t)/t=c_2$ almost surely for each $i \in \Z$. In general, $c_1$ and $c_2$ will not coincide. We know that 
$c_2 \ge W(F)$ since, by Fatou's Lemma, 
\begin{align*}
c_2&=\E \limsup_{t \to \infty}u_0(t)/t=F-\E\liminf_{t \to \infty}\big(F-u_0(t)/t\big)\\
&\ge F- \liminf_{t \to \infty}\E\big(F-u_0(t)/t\big)\ge W(F).
\end{align*}
We conjecture that we also have $c_1 \ge W(F)$.
\end{remark}

\section{Discretization of the continuum problem} \label{sec:cont}

We now return to the continuum problem. In the first part we prove the statement about the expected value of the velocity
employing a discretization of the problem. In the second part we extend the almost sure result from subsection~\ref{subsec:almost} to
the continuum model.

\subsection{Proof of the main theorem}
For the discretization, we largely rely on the ideas put forward by Coville, Dirr and Luckhaus in~\cite{Coville:2010p1074}. As seen there, we first introduce a modified problem, basically turning off the driving force
in the rows where the obstacles lie. Let $A := \R \setminus \{ \bigcup_{i\in\Z} (i-\delta, i+\delta)\}$. We restrict
$F$ to act on the set $A$. Let thus $\tilde{u} \colon \R\times [0,\infty)\times \Omega$ solve the modified problem
\begin{align}
\label{eq:modified}
\tilde{u}_t(x,t,\omega) &=  \tilde{u}_{xx}(x,t,\omega) - f(x,\tilde{u}(x,t,\omega),\omega) + F\chi_{A}(x) \\
\tilde{u}(x,0,\omega) &= 0 \nonumber
\end{align}
for $t\ge0$, $x\in\R$, $\omega \in \Omega$. Existence and uniqueness of classical solutions
for both the original problem~\eqref{eq:governing} as well as for the modified problem are proved 
in~\cite{Coville:2010p1074}, Lemma 3.2. Since we have
$$
- f(x,y,\omega) + F\chi_{A}(x) \le - f(x,y,\omega) + F \quad \textrm{for all} \quad x,y\in\R, \omega \in \Omega,
$$
it follows that
\begin{equation} \label{eq:modified_expectation}
\E  U(t,\omega) = \E \int_0^1 u(\xi, t, \omega) \dd \xi \ge \E  \tilde{U}(t,\omega) := \E \int_0^1 \tilde{u}(\xi, t, \omega) \dd \xi.
\end{equation}
by the comparison principle for parabolic equations~\cite{Nirenberg:1953p1315}.
It is thus sufficient to prove Theorem~\ref{thm:cont} replacing $u$ with a solution of the modified problem.

The following proposition shows that the velocity of a solution of~\eqref{eq:modified} will be non-negative
for all times.
\begin{proposition} \label{prop:pos_dudt}
Let $\tilde{u}$ be a solution of~\eqref{eq:modified}. We have
$$
\tilde{u}_t(x,t,\omega) \ge 0 \quad \textrm{for all} \quad t>0, x\in\R, \omega\in\Omega.
$$
\end{proposition}
\begin{proof}
Since there are no obstacles on the line $\{y=0\} \subset \R^2$, we have $\tilde{u}_t(x,0,\omega) \ge 0$ for all
$x\in \R$, $\omega \in \Omega$. In fact, there even exists $\epsilon>0$, so that $\tilde{u}_t(x,t,\omega) > 0$
for all $t\in (0,\epsilon)$, $x\in\R$, $\omega \in \Omega$, as one can easily see from the solution of the parabolic equation which is still linear for sufficiently small time $t$. Assume now that the proposition is untrue. Due to the
fact that the solution $\tilde{u}$ is classical, there would have to exist a minimal $t_0>0$ and $x_0\in\R$, $\omega_0\in\Omega$ such that
$$
\tilde{u}_t(x_0,t_0,\omega_0) = 0.
$$
and $\epsilon>0$ so that $\tilde{u}_t(x_0,t_0+\delta,\omega_0) < 0$ for all $\delta \in (0,\epsilon)$.
Differentiating the equation with respect to $t$ at this point yields\footnote{This additional time derivative might not be smooth,
but the argument holds unchanged by considering the equation in the sense of viscosity solutions.}
$$
\tilde{u}_{tt}(x_0,t_0,\omega_0) =  \tilde{u}_{xxt}(x_0,t_0,\omega_0) - 
f_y(x_0,\tilde{u}(x_0,t_0,\omega_0),\omega_0)\tilde{u}_{t}(x_0,t_0,\omega_0).
$$
The $f$-term vanishes due to the assumptions on $t_0$, $x_0$, and $\omega_0$. We note that
due to the fact that, by the assumption that $t_0$ is the first time that $u_t$ is becoming
negative anywhere,  $\tilde{u}_t(x,t_0,\omega_0) \ge 0$ for all $x\in\R$.  Thus, $\tilde{u}_{tt}(x_0,t_0,\omega_0)$ is positive contradicting the negativity of $\tilde{u}_{t}$.
\end{proof}

Next, we define an associated discrete process to the evolution equation~\eqref{eq:modified}. Let 
\begin{equation} \label{eq:discretized}
\hat{u}_i := \tilde{u}(i-\delta) + 2\delta \tilde{u}_x(i-\delta) \quad \textrm{for $i\in\Z$}.
\end{equation}
For notational simplicity, we have omitted the time- and $\omega$-dependence of the terms.

The next step is to use the estimates on the discrete Laplacian found in~\cite{Coville:2010p1074} and adapt them to our case of an additional positive velocity on the right hand side of the equation.
\begin{proposition} \label{prop:disc_lap_est}
We have, for $\tilde{u}_i$ defined as in~\eqref{eq:discretized}, $\delta<1/2$ as in equation~\eqref{eq:def_delta},
\begin{align}
\hat{u}_{i-1} -2 \hat{u}_i + \hat{u}_{i+1} \le \;& (1+2\delta)\left[\tilde{u}_x(i+\delta) - \tilde{u}_x(i-\delta)-\int_{i-\delta}^{i+\delta} \tilde{u}_t(\xi) \dd\xi \right]   \nonumber\\
&- (1-2\delta)F+2(1+\delta)\int_{i-1-\delta}^{i+1-\delta}  \tilde{u}_t(\xi) \dd\xi  \label{eq:lap_est}.
\end{align}
for each $i\in\Z$. The estimate holds for any time and on all of $\Omega$.
\end{proposition}
\begin{proof}
For some $i\in\Z$ let 
\begin{equation} \label{eq:vdef}
\tilde{v}(x) := \tilde{u}(x) - \int_{i-1-\delta}^{x}\int_{i-1-\delta}^{y} \tilde{u}_t(\xi) \dd \xi \dd y
\end{equation}
  for $x \in [i-1-\delta, i+1-\delta]$. We see that $\tilde{v}(x)$ solves 
$$
\tilde{v}_{xx} = \left\{ \begin{array}{ll} -F &\textrm{on $(i-1+\delta, i-\delta) \cup (i+\delta, i+1-\delta)$} \\
 f(x,\tilde{u}(x),\omega) &\textrm{otherwise}. 
 \end{array}\right.
$$
To this function, we apply the estimate of the discrete Laplacian found in~\cite{Coville:2010p1074}, Lemma 4.1, to obtain
\begin{equation} \label{eq:disc_est_novel}
\hat{v}_{i-1} -2 \hat{v}_i + \hat{v}_{i+1} \le  (1+2\delta)\left[\tilde{v}_x(i+\delta) - \tilde{v}_x(i-\delta)\right] - (1-2\delta)F.
\end{equation}
Here, $\hat{v}$ is discretized in the same way as $\hat{u}$, i.e, $\hat{v}_j = \tilde{v}(j-\delta) + 2\delta \tilde{v}_x(j-\delta)$, 
$j \in \{i-1,i,i+1\}$.

The next step is to estimate the difference between $\hat{v}$ and $\hat{u}$. First, we note that
\begin{equation} \label{eq:force_mod}
\left[\tilde{v}_x(i+\delta) - \tilde{v}_x(i-\delta)\right]  = \left[\tilde{u}_x(i+\delta) - \tilde{u}_x(i-\delta)-\int_{i-\delta}^{i+\delta} \tilde{u}_t(\xi) \dd\xi \right].
\end{equation}
It is clear that $\hat{v}_{i-1} = \hat{u}_{i-1}$. Furthermore, we have, by the non-negativity of $\tilde{u}_t$
(see Proposition~\ref{prop:pos_dudt}), that
$$
\tilde{u}(i-\delta) \ge \tilde{v}(i-\delta)
$$
and
$$
\tilde{u}_x(i-\delta) \ge \tilde{v}_x(i-\delta),
$$
which yields, by the definition of the discretization,
$$
\hat{u}_i \ge \hat{v}_i,
$$
and thus
\begin{equation} \label{eq:first_lap}
\hat{u}_i - \hat{u}_{i-1} \ge  \hat{v}_i -  \hat{v}_{i-1}.
\end{equation}
For the second term in the discrete Laplacian, we first estimate, by~\eqref{eq:vdef}, and again using positivity of $\tilde{u}_t$,
\begin{align*}
\tilde{u}(i+1-\delta)  &= \tilde{v}(i+1-\delta)  + \int_{i-1-\delta}^{x}\int_{i-1-\delta}^{x} \tilde{u}_t(\xi) \dd \xi \dd x \\
&\le \tilde{v}(i+1-\delta) + 2 \int_{i-1-\delta}^{i+1-\delta}  \tilde{u}_t(\xi) \dd \xi.
\end{align*}
In addition, we see that 
$$
u_x(i+1-\delta) \le v_x(i+1-\delta) + \int_{i-1-\delta}^{i+1-\delta}  \tilde{u}_t(\xi) \dd \xi
$$
and find, combining those two estimates and using again the definition of the discretization,
\begin{equation} \label{eq:second_lap}
\hat{u}_{i+1} - \hat{u}_i \le \hat{v}_{i+1} - \hat{v}_i + 2(1+\delta) \int_{i-1-\delta}^{i+1-\delta}  \tilde{u}_t(\xi) \dd \xi.
\end{equation}
Now, subtracting~\eqref{eq:first_lap} from~\eqref{eq:second_lap}, we get 
\begin{equation} \label{eq:lap_est2}
\hat{u}_{i-1} -2 \hat{u}_i + \hat{u}_{i+1} \le \hat{v}_{i-1} -2 \hat{v}_i + \hat{v}_{i+1} + 2(1+\delta) \int_{i-1-\delta}^{i+1-\delta}  \tilde{u}_t(\xi) \dd \xi.
\end{equation}
The proposition follows by inserting~\eqref{eq:disc_est_novel} and~\eqref{eq:force_mod} into~\eqref{eq:lap_est2}.
\end{proof}
\begin{corollary} \label{cor:vel_est_pos}
Using the non-negativity of $\tilde{u}_t$ shown in Proposition~\ref{prop:pos_dudt} we can deduce from Proposition~\ref{prop:disc_lap_est} that
\begin{align*}
2(1+\delta)\int_{i-1-\delta}^{i+1-\delta} \tilde{u}_t(\xi) \dd\xi \ge\;& \Big(\hat{u}_{i-1} -2 \hat{u}_i + \hat{u}_{i+1} \\
&-  (1+2\delta)\left[\tilde{u}_x(i+\delta) - \tilde{u}_x(i-\delta)-\int_{i-\delta}^{i+\delta} \tilde{u}_t(\xi) \dd\xi \right]  \\
&+ (1-2\delta)F \Big)^+,
\end{align*}
i.e., the integrated velocity can be estimated by the positive part of the discretized problem.
\end{corollary}

Next, we come to estimate the effect of the obstacles. We procede as in~\cite{Coville:2010p1074}. Let
$$
k(i) := \tilde{u}_x(i-\delta) - \tilde{u}_x(i+\delta).
$$
The following proposition gives an estimate for $k$ in terms of the obstacles passed by the function $\tilde{u}$.
\begin{proposition} \label{prop:obst_est}
Given $\tilde{u}(x) = \tilde{u}(x,t,\omega)$, solution of the evolution equation~\eqref{eq:modified} at fixed time and fixed $\omega \in\Omega$,
let $i\in \Z$, $M := \max\{ \abs{\tilde{u}_x(i-\delta)}, \abs{\tilde{u}_x(i+\delta)} \}$. We then have
$$
k(i) \le \frac{18\delta}{M} \sum_{\hat{u}_i -4\delta M \le j \le \hat{u}_i + 4\delta M  } f_{ij}(\omega) 
+ \int_{i-\delta}^{i+\delta} \tilde{u}_t(\xi) \dd \xi.
$$
\end{proposition}
\begin{proof}
This follows immediately from~\cite{Coville:2010p1074}, Lemma 4.2, after subtracting and adding the effect of
the positive time derivative on a snapshot of the function $\tilde{u}$. The proof in~\cite{Coville:2010p1074} only uses convexity of the 
function $\tilde{u}$ inside $(i-\delta, i+\delta)$, which also holds in our case.
\end{proof}
We now need a uniform estimate on some exponential moment of the average of the random variables in the estimate of $k$. The following 
proposition provides this result.
\begin{proposition} \label{prop:av_exp_estimate}
Let $X_1, X_2, \dots$ be real valued iid random variables such that there exists $\lambda >0$ with
$$
\E \exp \{ \lambda X_1 \} = \beta <\infty.
$$
Then we have
$$
\E \exp \{ \tilde{\lambda} \sup_{N\in\N} \frac{1}{N} \sum_{j=1}^N X_j \} \le \tilde{\beta}
$$
for $0 < \tilde{\lambda} < \lambda$ and
$$
\tilde{\beta} =\inf_{c>\exp\{(\log \beta)\tilde{\lambda}/\lambda\}}\left(c+\int_c^\infty  
\frac{ \beta \ee^{-\frac{\lambda}{\tilde{\lambda}}\log x }}{1-\beta\ee^{- \frac{\lambda}{\tilde{\lambda}} \log x }} \dd x\right) < \infty.
$$ 
\end{proposition}
\begin{proof}
Let $S := \sup_{N\ge 1} \frac{1}{N} \sum_{j=1}^N X_j$. It is clear from the strong law of large numbers that this supremum is finite almost
surely. We have
\begin{align*}
\P\left( S > u\right) &\le \sum_{N=1}^\infty \P\left( \frac{1}{N} \sum_{j=1}^N X_j \ge u\right) \\
&\le  \sum_{N=1}^\infty \P\left( \ee^{\lambda \sum_{j=1}^N X_j} \ge \ee^{\lambda N u}\right) \\
&\le \sum_{N=1}^\infty \ee^{-\lambda Nu}\left( \E \ee^{\lambda X_1 } \right)^N \\
&= \sum_{N=1}^\infty \ee^{-N \left( \lambda u - \log \E \ee^{\lambda X_1} \right) }.
\end{align*}
Markov's inequality as well as independence and the identical distribution of the $X_j$ was used in the last inequality. Setting
$I(u) := \lambda u - \log \E \ee^{\lambda X_1} =  \lambda u - \log \beta$, we find that
$$
\P\left( S> u\right) \le \sum_{N=1}^\infty \ee^{-NI(u)} = \frac{\ee^{-I(u)}}{1-\ee^{-I(u)}}
$$
for $u$ so large that the series converges, i.e, we have $I(u)>0$. Setting $0 < \tilde{\lambda} < \lambda$ yields
\begin{align*}
\E \ee^{\tilde{\lambda} S} &= \int_0^\infty \P \left(\ee^{\tilde{\lambda}S} > x\right) \dd x = 
\int_0^\infty \P\left( S > \frac{1}{\tilde{\lambda}} \log x\right) \dd x \\
&\le c + \int_c^\infty  \frac{\ee^{-I( \frac{1}{\tilde{\lambda}} \log x)}}{1-\ee^{-I( \frac{1}{\tilde{\lambda}} \log x)}}  \dd x.
\end{align*}
This holds for any $c>\exp\{(\log \beta)\tilde{\lambda}/\lambda\}$, since then $I(\frac{1}{\tilde{\lambda}} \log c)>0$.
\end{proof}
\begin{remark}
A similar result has been derived in~\cite{Siegmund:1969p1779}, albeit without an explicit quantitative estimate.
\end{remark}

We can now proceed to prove the main result.\\
\begin{proof}[ of Theorem~\ref{thm:cont}]
The proof is divided into three steps. We first use Proposition~\ref{prop:av_exp_estimate} to find a suitable set of random variables
to dominate the effect of the precipitates. Then we show that the discretization of a solution to the continuous initial value
problem~\eqref{eq:modified} has to remain bounded from below. Finally we use the obstacles from the first step in a discretized 
problem and are able to estimate the integrated velocity.
\\[3mm]
{\bf Step 1:} Set
$$
g_{ij}(\omega) := 1+(1+2\delta)\sup_{M \in \N} \frac{36 \delta}{M} \sum_{j-4\delta M \le l \le j+4\delta M} f_{il}(\omega).
$$
From Proposition~\ref{prop:obst_est} it is clear (note that $M\ge1/2$ if $k\ge 1$) that $g_{ij} \ge k(i)$ for $j=\lceil \hat{u}_i-1/2 \rceil$, i.e., $j$ is $\hat{u}_i$
rounded to the nearest integer.
Using the bound $\delta < 1/2$ and counting the number of summands we find that
$$
g_{ij} \le 1+ 180 \sup_{N\in \N\cup \{0\}} \frac{1}{2N+1} \sum_{k=j-N}^{j+N} f_{ik}.
$$
Since the $f_{ik}$ are independent and identically distributed and we have by Assumption~\ref{ass:f} that
$\E e^{\lambda f_{ik}} =\beta <\infty$ it is now possible to give a bound on an exponential moment of $g_{ij}$ using
Proposition~\ref{prop:av_exp_estimate}. Namely, we have
$$
\E \exp \{ \tilde{\lambda} g_{ij} \} \le \tilde{\beta}
$$
with $0 < \tilde{\lambda} < \lambda$ and 
$$
\tilde{\beta}:= \ee^{\tilde{\lambda}} \inf_{c>\exp\{(\log\beta)\frac{180\tilde{\lambda}}{\lambda}\}}\left(c+\int_c^\infty  \frac{ \beta \ee^{-\frac{\lambda}{180\tilde{\lambda}}\log x }}{1-\beta\ee^{- \frac{\lambda}{180\tilde{\lambda}} \log x }} \dd x\right) < \infty.
$$
Now let 
$$
\tilde{f}_i(y,\omega) := g_{i\lceil y-1/2\rceil}(\omega),
$$
i.e, simply evaluate $g_{ij}$ at $j=\lceil y-1/2 \rceil$, so $j$ is $y$ rounded to the nearest integer. It follows that
\begin{equation} \label{eq:dominate_obst}
 \tilde{u}_x(i-\delta) - \tilde{u}_x(i+\delta)- \int_{i-\delta}^{i+\delta} \tilde{u}_t(\xi) \dd \xi \le \tilde{f}_i(\hat{u},\omega)   .
\end{equation}
\\[3mm]
{\bf Step 2:} Now let $\tilde{u}(x,t_0,\omega)$ be a solution of the evolution equation~\eqref{eq:modified} at time $t_0$ with fixed parameter $F$
and let $\hat{u}_i$ for $i \in \Z$ be its discretization according to equation~\eqref{eq:discretized}. We claim that there exists $C\in\R$, so that
$$
\hat{u}_i \ge C \quad.
$$
Indeed, we have $t_0 F\ge \tilde{u}\ge0$, so $\inf_{i\in \Z} (\tilde{u}(i+1-\delta) - \tilde{u}(i-\delta)) \ge -t_0F$. But we also have
that $\tilde{u}_{xx} \ge - F$, due to the non-negativity of $\tilde{u}_t$ and the non-negativity of $f$. It follows that 
\begin{align*}
-t_0F \le \tilde{u}(i+1-\delta) - \tilde{u}(i-\delta) &= \int_{i-\delta}^{i+1-\delta} \tilde{u}_x(\xi) \dd\xi \\
&\le \sup_{\xi \in [i-\delta,i+1-\delta]} \tilde{u}_x(\xi) \\
&\le \tilde{u}_x(i+1-\delta) +F.
\end{align*}
Thus, $\hat{u}_{i+1} \ge \tilde{u}(i+1-\delta)  -2\delta (1+t_0)F \ge -2\delta (1+t_0)F$ for all $i\in\Z$. 
\\[3mm]
{\bf Step 3:} Now we can finally apply Lemma~\ref{lem:vel_est} to the discrete process $\hat{u}_i$. It is clear that the estimate
in the Lemma holds for any sequence that is uniformly bounded from below, the bound zero in the proof is arbitrary. We thus find that
$$
W(F) \le \lim_{n \to \infty}\frac 1n \sum_{i=0}^{n-1} \Big(\hat{u}_{i-1}+\hat{u}_{i+1}-2\hat{u}_i - \tilde{f}_i(\hat{u}_i,\omega) +F \Big)^+,
$$
with $W$ from Theorem~\ref{main} using the parameters $\tilde{\lambda}$ and $\tilde{\beta}$ from Step 1.
Using Corollary~\ref{cor:vel_est_pos} and equation~\eqref{eq:dominate_obst}, we see that
\begin{align*}
W((1-2\delta)F) &\le  \lim_{n \to \infty}\frac 1n \sum_{i=0}^{n-1} \Big(\hat{u}_{i-1} -2 \hat{u}_i + \hat{u}_{i+1} - \tilde{f}_i(\hat{u}_i,\omega) + (1-2\delta)F \Big)^+ \\
&\le   \lim_{n \to \infty}\frac 1n \sum_{i=0}^{n-1} \Big(\hat{u}_{i-1} -2 \hat{u}_i + \hat{u}_{i+1} \\
&\qquad  - \big( \tilde{u}_x(i-\delta) - \tilde{u}_x(i+\delta)- \int_{i-\delta}^{i+\delta} \tilde{u}_t(\xi) \dd \xi \big)  \\
&\qquad + (1-2\delta)F \Big)^+\\
&\le \lim_{n \to \infty}\frac 1n \sum_{i=0}^{n-1} 2(1+\delta) \int_{i-1-\delta}^{i+1-\delta} \tilde{u}_t(\xi) \dd\xi\\
&\le 4(1+\delta)\, \E \frac{\dd}{\dd t}\tilde{U}(t).
\end{align*}
Integrating $\frac{\dd}{\dd t}\tilde{U}(t)$ and using~\eqref{eq:modified_expectation} proves the theorem.
\end{proof}

\subsection{Almost sure statements in the continuum model} \label{subsec:almost_cont}
It is possible to show the analog of Proposition~\ref{prop:diff_est} in the continuum model.
\begin{proposition} \label{prop:diff_conv}
Consider a solution of the model~\eqref{eq:governing}, again with the relaxed assumption that $f_{ij}$ are
nonnegative random variables. Then we have $\lim_{t\to\infty} \frac{u(t,x_2)-u(t,x_1)}{t} = 0$ for
all $x_1,x_2 \in \R$ and all $\omega\in\Omega$. Furthermore, the convergence is uniform for $x_1, x_2$ chosen on a compact
interval.
\end{proposition}
\begin{proof}
Let, without loss of generality $x_1 < x_2$ and let $u_j:=u(x_1+(j-1)(x_2-x_1))$ for $j\in\Z$. Again define
$$
H(t) := u_{j-1}(t) + u_{j+1}(t) - 2u_j(t) + 2(x_2-x_1)^2 F, \quad \textrm{for $j\in\Z$}.
$$
We have $H(t)\ge 0$. Indeed, assume that there exist $t_0\ge0, j_0\in\Z$ with $H(t_0)<0$, then at some point $\xi \in (x_1+(j_0-1)(x_2-x_1), x_1+(j_0+1)(x_2-x_1))$ we have
$u_{xx}(t_0, \xi) +F <0$. At this point, however, the propagation velocity of the interface would be negative,
which violates the comparison principle (see Proposition~\ref{prop:pos_dudt}). The rest of the proof continues
as the proof of the discrete version of the proposition.
\end{proof}
\begin{remark}
The analog of Remark~\ref{rem:limsup} also holds, i.e., under the assumptions of Theorem~\ref{thm:cont} there exist deterministic
numbers $0 \le c_1 \le c_2 \le F$ such that 
\begin{align*}
\liminf_{t \to \infty}u(t,x)/t&=c_1 \quad \textrm{and}  \\
\limsup_{t \to \infty}u(t,x)/t = \limsup_{t\to\infty} \inf_{\xi\in K} u(t,\xi)/t&=c_2
\end{align*} almost surely, for each $x \in \R$ and for any non empty compact set $K\subset \R$ . This follows by using the ergodicity of the process $\int_j^{j+1} u(\xi,t)/t \dd\xi$, $j\in\Z$ 
and by using the uniform convergence from Proposition~\ref{prop:diff_conv}.

Furthermore, we have $c_2 \ge V(F)$. Indeed, integrating over a spatial period and using uniform convergence, then using
Fubini's theorem yields
\begin{align*}
c_2&=\E\int_0^1 \limsup_{t \to \infty}u(t,\xi)/t \dd \xi = F-\E\liminf_{t \to \infty}\left(F-\int_0^1u(t,\xi)/t \dd \xi \right)\\
&\ge F- \liminf_{t \to \infty}\E\left(F-\int_0^1u(t,\xi)/t \dd \xi \right)\ge V(F).
\end{align*}
The question whether $c_1>0$ remains open also here.
\end{remark}

\section{Conclusion and open problems} \label{sec:conclusion}
We have shown that interfaces in random media can cover a finite area per unit time on average, even if there is 
no uniform upper bound to the strength of obstacles. This is an extension to~\cite{Coville:2010p1074}, where
non-existence of a stationary solution was shown under the same assumptions. Many questions remain open, however. 
Amongst those are the conjecture stated in the previous section, namely whether an almost sure statement can be made about
the inferior limit at a fixed point, even in the discrete problem.

Also open is the extension of the theorem to more than one dimension, as well as the question whether there exists a non-trivial
interval so that if $F\in [F_0,F_1]$ we do not have a finite velocity, but also no stationary solution.

\section*{Acknowledgements}
We are delighted to acknowledge useful discussions with Nicolas Dirr and are grateful for support from the DFG funded research group 
`Analysis and Stochastics in Complex Physical Systems.'

%\bibliographystyle{amsalpha}
%\bibliography{/Users/pwd/Research/Papers/all}

\newcommand{\etalchar}[1]{$^{#1}$}
\providecommand{\bysame}{\leavevmode\hbox to3em{\hrulefill}\thinspace}
\providecommand{\MR}{\relax\ifhmode\unskip\space\fi MR }
\providecommand{\MRhref}[2]{%
  \href{http://www.ams.org/mathscinet-getitem?mr=#1}{#2}
}
\providecommand{\href}[2]{#2}

\end{document}